\documentclass[12pt,twoside,a4paper]{article}
\usepackage{bm}
\usepackage{braket}
\usepackage{graphicx}
\usepackage{graphics}
\usepackage{theorem}
\usepackage{amsmath}
\usepackage{amssymb,mathrsfs}
\usepackage{latexsym}
\usepackage{cases}
\usepackage[dvipdfmx,bookmarksnumbered,colorlinks,linkcolor=blue,urlcolor=blue,citecolor=blue]{hyperref}
\usepackage{color}
\usepackage{algorithm}
\usepackage{algpseudocode}
\usepackage{caption}
\newfloat{procedure}{H}{loa}
\floatname{procedure}{Procedure} 
\newtheorem{THM}{Theorem}[section]
\newtheorem{PROP}[THM]{Proposition}

\newtheorem{ASSU}[THM]{Assumption}

\newtheorem{REM}[THM]{Remark}
\newtheorem{DEF}[THM]{Definition}

\newtheorem{hm-cond}{Condition}
\usepackage[dvipdfmx]{hyperref}
\hypersetup{
bookmarksnumbered,colorlinks=false,linkcolor=blue,urlcolor=blue,citecolor=blue
}

\numberwithin{equation}{section}  
\newcommand{\vc}{\bm}

\makeatletter
\DeclareRobustCommand\widecheck[1]{{\mathpalette\@widecheck{#1}}}
\def\@widecheck#1#2{%
    \setbox\z@\hbox{\m@th$#1#2$}%
    \setbox\tw@\hbox{\m@th$#1%
       \widehat{%
          \vrule\@width\z@\@height\ht\z@
          \vrule\@height\z@\@width\wd\z@}$}%
    \dp\tw@-\ht\z@
    \@tempdima\ht\z@ \advance\@tempdima2\ht\tw@ \divide\@tempdima\thr@@
    \setbox\tw@\hbox{%
       \raise\@tempdima\hbox{\scalebox{1}[-1]{\lower\@tempdima\box
\tw@}}}%
    {\ooalign{\box\tw@ \cr \box\z@}}}
\makeatother
\makeatletter
\newif\if@borderstar
\def\bordermatrix{\@ifnextchar*{%
 \@borderstartrue\@bordermatrix@i}{\@borderstarfalse\@bordermatrix@i*}%
}
\def\@bordermatrix@i*{\@ifnextchar[{\@bordermatrix@ii}{\@bordermatrix@ii[()]}}
\def\@bordermatrix@ii[#1]#2{%
\begingroup
 \m@th\@tempdima8.75\p@\setbox\z@\vbox{%
 \def\cr{\crcr\noalign{\kern 2\p@\global\let\cr\endline }}%
 \ialign {$##$\hfil\kern 2\p@\kern\@tempdima & \thinspace %
  \hfil $##$\hfil && \quad\hfil $##$\hfil\crcr\omit\strut %
  \hfil\crcr\noalign{\kern -\baselineskip}#2\crcr\omit %
  \strut\cr}}%
 \setbox\tw@\vbox{\unvcopy\z@\global\setbox\@ne\lastbox}%
 \setbox\tw@\hbox{\unhbox\@ne\unskip\global\setbox\@ne\lastbox}%
 \setbox\tw@\hbox{%
  $\kern\wd\@ne\kern -\@tempdima\left\@firstoftwo#1%
  \if@borderstar\kern 2pt\else\kern -\wd\@ne\fi%
 \global\setbox\@ne\vbox{\box\@ne\if@borderstar\else\kern 2\p@\fi}%
 \vcenter{\if@borderstar\else\kern -\ht\@ne\fi%
  \unvbox\z@\kern -\if@borderstar2\fi\baselineskip}%
 \if@borderstar\kern-2\@tempdima\kern2\p@\else\,\fi\right\@secondoftwo#1 $%
 }\null \;\vbox{\kern\ht\@ne\box\tw@}%
\endgroup
}
\makeatother

\newcommand{\ol}{\overline}
\newcommand{\ool}[1]{\overline{\overline{\bm{#1}}}}

\newcommand{\down}[2]{\smash{\lower#1\hbox{#2}}}
\newcommand{\up}[2]{\smash{\lower-#1\hbox{#2}}}

\newcommand{\dm}{\displaystyle}

\newcommand{\vmin}{\wedge}

\newcommand{\EE}{\mathbb{E}}
\newcommand{\PP}{\mathbb{P}}

\newcommand{\one}{\mbox{$1$}\hspace{-0.25em}{\rm l}}

\newcommand{\calL}{\mathcal{L}}

\newcommand{\calS}{\mathcal{S}}
\newcommand{\calX}{\mathcal{X}}

\newcommand{\bbL}{\mathbb{L}}
\newcommand{\bbM}{\mathbb{M}}
\newcommand{\bbN}{\mathbb{N}}
\newcommand{\bbR}{\mathbb{R}}
\newcommand{\bbS}{\mathbb{S}}

\newcommand{\bbZ}{\mathbb{Z}}

\newcommand{\dd}[1]{\if#11 1\!\!1
\else {\if#1C I\!\!\!C
\else {\if#1G I\!\!\!G
\else {\if#1J J\!\!\!J
\else {\if#1S S\!\!\!S
\else {\if#1Z Z\!\!\!Z
\else {\if#1Q O\!\!\!\!Q
\else I\!\!#1
\fi}
\fi}
\fi}
\fi}
\fi}
\fi}
\fi}
\makeatletter
\def\widebar{\accentset{{\cc@style\underline{\mskip10mu}}}}
\def\Widebar{\accentset{{\cc@style\underline{\mskip8mu}}}}
\makeatother
\newcommand{\wbartil}[1]{\if#1L \widebar{\hspace{-0.12zw}\widetilde{#1}}
\else {\if#1M \widebar{\widetilde{\!M\!}}
\else {\if#1W \widebar{\widetilde{\!W\!}}
\else {\if#1U \widebar{\hspace{-0.03zw}\widetilde{#1}}
\else {\if#1V \widebar{\hspace{-0.03zw}\widetilde{#1}}
\else {\if#1Y \widebar{\hspace{-0.0zw}\widetilde{#1}}
\else \,\widebar{\!\widetilde{#1}}
\fi}
\fi}
\fi}
\fi}
\fi}
\fi}
\makeatletter
\def\eqnarray{\stepcounter{equation}\let\@currentlabel=\theequation
\global\@eqnswtrue
\global\@eqcnt\z@\tabskip\@centering\let\\=\@eqncr
$$\halign to \displaywidth\bgroup\@eqnsel\hskip\@centering
$\displaystyle\tabskip\z@{##}$&\global\@eqcnt\@ne
\hfil$\;{##}\;$\hfil
&\global\@eqcnt\tw@ $\displaystyle\tabskip\z@{##}$\hfil
\tabskip\@centering&\llap{##}\tabskip\z@\cr}
\makeatother

\makeatother

\begin{document}\thispagestyle{empty}

\hfill

\vspace{-10mm}

{\large{\bf
\begin{center}
A SUBGEOMETRIC CONVERGENCE FORMULA FOR TOTAL-VARIATION ERROR OF THE LEVEL-INCREMENT TRUNCATION APPROXIMATION OF M/G/1-TYPE MARKOV CHAINS%
%
%
%
%
\end{center}
}
}

\begin{center}
{
\begin{tabular}[h]{cc}
Katsuhisa Ouchi\footnotemark[2] & Hiroyuki Masuyama\footnotemark[3]                        \\ 
\textit{Kyoto University}&\textit{Tokyo Metropolitan University}\\ 
\end{tabular}
\footnotetext[2]{E-mail: o-uchi@sys.i.kyoto-u.ac.jp}
\footnotetext[3]{E-mail: masuyama@tmu.ac.jp}
}

\bigskip
\medskip

{\small
\textbf{Abstract}

\medskip

\begin{tabular}{p{0.85\textwidth}}
This paper considers the level-increment (LI) truncation approximation of M/G/1-type Markov chains. The LI truncation approximation is usually used to implement Ramaswami's recursion for the stationary distribution in M/G/1-type Markov chains. The main result of this paper is a subgeometric convergence formula for the total-variation distance between the stationary distribution and its LI truncation approximation.
\end{tabular}
}
\end{center}

\begin{center}
\begin{tabular}{p{0.90\textwidth}}
{\small

\medskip

}
\end{tabular}

\end{center}

\section{Introduction}
\label{sec:Intro}

The level-increment (LI) truncation approximation is one of simple and practical approaches to implementing Ramaswami's recursion for the stationary distribution vector in M/G/1-type Markov chains (see, e.g., the introduction of \cite{Ouchi-Masuyama21-2}). The LI truncation approximation modifies the original M/G/1-type transition probability matrix into another M/G/1-type one with bounded level increments (with jump sizes truncated at a particular level). In the corresponding modified M/G/1-type Markov chain, Ramaswami's recursion does not include infinite sums and thus is implementable on a computer.

Even though the LI truncation approximation would be the simplest and most practical one of the approximations for implementing Ramaswami's recursion, there has not yet been sufficient research on its error evaluation. Ouchi and Masuyama presented the convergence formulas for the {\it level-wise (i.e., not the whole)} difference between the original stationary distribution vector and its LI truncation approximation, assuming the equilibrium level-increment distribution (in steady-state) is light-tailed \cite{Ouchi-Masuyama21-2} and long-tailed \cite{Ouchi-Masuyama21}. Although such a convergence formula identifies the convergence speed of the level-wise error of the LI truncation approximation, it does not necessarily identify the convergence speed of the LI truncation approximation as a whole. In this sense, the error evaluation of the level-wise difference is not enough compared to the error evaluation by norm, such as the total variation norm. In addition, as far as we know, there have been no studies on the error evaluation of the whole LI truncation approximation by norm, even in an M/G/1-type Markov chain with a single phase.

The purpose of this paper is to derive a subgeometric convergence formula for the {\it total-variation} error of the LI truncation approximation. The rest of this paper consists of three sections. Section~\ref{sec:preliminary_results} introduces preliminary results related to the M/G/1-type Markov chain and its LI truncation approximation. Section~\ref{sec:subgeometric} presents the main result of this paper, that is, the subgeometric convergence formula for the total-variation error of the LI truncation approximation. Section~\ref{sec:concluding} contains concluding remarks.

\section{Preliminaries}
\label{sec:preliminary_results}
This section provides preliminary results on the M/G/1-type Markov chain and its LI truncation approximation. We first introduce the M/G/1-type Markov chain and a well-known sufficient condition for the existence of the stationary distribution. We then describe Ramswami's recursion for the stationary distribution. After that, we introduce the LI truncation approximation as a practical approach to implementing Ramswami's recursion.

We begin with the definition of the M/G/1-type Markov chain.
To this end, let $\bbN = \{1,2,3,\dots\}$ and $\bbZ_+ = \bbN \cup \{0\}$. Let $\bbL_k = \{k\} \times \mathbb{M}_{k \vmin 1}$ for $k \in \bbZ_+$, where
\begin{alignat*}{2}
\mathbb{M}_0 &= \{1,2,\ldots, M_0\} \subset \bbN,
&\qquad
\mathbb{M}_1 &= \{1,2,\ldots,M_1\} \subset \bbN,
\end{alignat*}
and $x \vmin y = \min(x,y)$ for $x,y \in (-\infty,\infty)$. We then define $\{(X_n, J_n); n \in \mathbb{Z}_+\}$ as a discrete-time Markov chain on state space $\bbS:=\cup_{k=0}^{\infty}\bbL_k$ with the transition probability matrix $\vc{P}$:
\begin{equation*}
\vc{P}=
\bordermatrix{
& \bbL_0
& \bbL_1
& \bbL_2
& \bbL_3
& \cdots \cr
\bbL_0
& \vc{B}(0)
& \vc{B}(1)
& \vc{B}(2)
& \vc{B}(3)
& \cdots \cr
\bbL_1
& \vc{B}(-1)
& \vc{A}(0)
& \vc{A}(1)
& \vc{A}(2)
& \cdots \cr
\bbL_2
& \vc{O}
& \vc{A}(-1)
& \vc{A}(0)
& \vc{A}(1)
& \cdots \cr
\bbL_3
& \vc{O}
& \vc{O}
& \vc{A}(-1)
& \vc{A}(0)
& \cdots
\cr
~\vdots
& \vdots
& \vdots
& \vdots
& \vdots
& \ddots
},
\end{equation*}
where $\vc{O}$ denotes the zero matrix, $\vc{B}(-1)$ denotes an $M_1 \times M_0$ nonnegative matrix, $\vc{B}(k)$, $k \in \bbZ_+$, denotes an $M_0 \times M_{k \vmin 1}$ nonnegative matrix, and  $\vc{A}(k)$, $k \ge -1$, denotes an $M_1 \times M_1$ nonnegative matrix. The subset $\bbL_k$ is called {\it level} $k$ and a state $(k, i) \in \bbL_k$ is called {\it phase} $i$ of level $k$. We refer to $\{(X_n, J_n)\}$ as the M/G/1-type Markov chain.

We now introduce the fundamental assumption of this paper.
\begin{ASSU} \label{assum1}
Let $\vc{e}$ denote a column vector consisting an appropriate number of ones, and let $\ol{\vc{m}}_A = \sum_{k = -1}^\infty k\vc{A}(k)\vc{e}$. The following hold: (i) the stochastic matrix $\vc{P}$ is irreducible;
(ii) $\vc{A}:=\sum_{k=-1}^\infty \vc{A}(k)$ is an irreducible stochastic matrix; (iii) $\ol{\vc{m}}_B:=\sum_{k = 1}^\infty k\vc{B}(k)\vc{e}$ is finite; and
(iv) $\sigma := \vc{\varpi} \ol{\vc{m}}_A < 0$, where $\vc{\varpi}$ denotes the  stationary distribution vector of $\vc{A}$.
\end{ASSU}

\smallskip

Assumption~\ref{assum1} ensures that $\vc{P}$ is irreducible and positive recurrent, which implies that $\vc{P}$ has the unique stationary distribution vector, denoted by $\vc{\pi} := (\pi(k, i))_{(k,i) \in \bbS}$ (see, e.g., \cite[Chapter~XI, Proposition~3.1]{Asmu03}). For later use, let $\vc{\pi}(k) := (\pi(k, i))_{i\in\bbM_{k\vmin 1}}$, $k \in \bbZ_+$, denote the level-wise subvector of $\vc{\pi}$. By definition,
\begin{align*}
\vc{\pi} = (\vc{\pi}(0),\vc{\pi}(1),\dots).
\end{align*}

We define some substochastic matrices, which are key components of Ramaswami's recursion for the stationary distribution vector $\vc{\pi} = (\vc{\pi}(0),\vc{\pi}(1),\dots)$. Let $\vc{G} := (G_{i, j})_{(i,j)\in(\bbM_1)^2}$ denote an $M_1 \times M_1$ matrix such that
\begin{align*}
G_{i, j} = \PP\left((X_{T_1}, J_{T_1})
= (1, j) \mid (X_0, J_0) = (2, i)\right), \qquad
i, j \in \mathbb{M}_1,
\end{align*}
where $T_k = \inf\{n \in \bbN: X_n = k\}$ for $k \in \bbZ_+$. The matrix $\vc{G}$ is called the $G$-matrix and is the limit of the following sequence $\{\vc{G}_n; n \in \bbZ_+\}$:
\begin{subnumcases}{\label{eqn:recursion_G} \vc{G}_n=}
\vc{O}, & \text{$n=0$,}
\\
\sum_{m = -1}^{\infty} \vc{A}(m) [\vc{G}_{n-1}]^{m+1},
 & \text{$n \in \bbN$,}
\end{subnumcases}
where $\vc{O}^0 = \vc{I}$. Assumption~\ref{assum1} (ii) and (iv) ensure that $\vc{G}$ is stochastic \cite[Theorem 2.3.1]{Neut89} and has the unique stationary distribution vector, denoted by $\vc{g}$ \cite[Proposition~2.1]{Kimu10}. Furthermore, let
\begin{align}
\vc{K}
&= \vc{B}(0) + \sum_{m=1}^{\infty} \vc{B}(m) \vc{G}^m,
\label{defn:K}
\\
\vc{\Phi}(0) &= \sum_{m=0}^\infty \vc{A}(m)\vc{G}^m.
\label{defn:Phi(0)}
\end{align}
Assumption~\ref{assum1} ensures that $\vc{K}$ is an irreducible stochastic matrix and has the unique stationary distribution vector, denoted by $\vc{\kappa}$ (see \cite[Theorem~3.1]{Sche90}). Assumption~\ref{assum1} also ensures that the Neumann series of $\vc{\Phi}(0)$ is convergent and thus $\sum_{m=0}^\infty [\vc{\Phi}(0)]^m =  (\vc{I} - \vc{\Phi}(0))^{-1}$ (see the proof of \cite[Theorem~2.1~(ii)]{Sche90}), where $\vc{I}$ denote the identity matrix. With the inverse matrix $(\vc{I} - \vc{\Phi}(0))^{-1}$, we define the following matrices:
\begin{subequations}\label{defn:R-matrices}
\begin{alignat}{2}
\vc{R}(k)
&=
\dm\sum_{m=0}^{\infty}
\vc{A}(k+m)\vc{G}^m (\vc{I}-\vc{\Phi}(0))^{-1}, & \qquad k &\in \bbN,
\label{defn-R(k)}
\\
\vc{R}_0(k)
&=
\sum_{m=0}^{\infty}
\vc{B}(k+m)\vc{G}^m(\vc{I}-\vc{\Phi}(0))^{-1}, & \qquad k &\in \bbN,
\label{defn-R_0(k)}
\\
\vc{R}_0 &= \sum_{m=1}^{\infty}\vc{R}_0(m),
\qquad
\vc{R} = \sum_{m=1}^{\infty}\vc{R}(m).
\label{defn:R_0-R}
\end{alignat}
\end{subequations}

We are ready to describe Ramaswami's recursion for $\{\vc{\pi}(k);k \in \bbZ_+\}$ (see \cite{Rama88,Sche90}). The sequence $\{\vc{\pi}(k);k \in \bbZ_+\}$ is determined by
\begin{subequations}\label{eq:Ramaswami}
\begin{align}
\vc{\pi}(0)
&= {
\vc{\kappa}
\over
\vc{\kappa}\vc{R}_0 (\vc{I} - \vc{R})^{-1} \vc{e}
},
\label{eq:Rama-pi(0)}
\\
\vc{\pi}(k)
&= \vc{\pi}(0)\vc{R}_0(k)
+ \sum_{\ell=1}^{k-1}\vc{\pi}(\ell)\vc{R}(k-\ell),
\qquad
k\in\bbN,
\end{align}
\end{subequations}
where $(\vc{I} - \vc{R})^{-1} = \sum_{m=0}^{\infty}\vc{R}^m$ holds if the M/G/1-type Markov chain $\{(X_{n},J_{n})\}$ is irreducible and positive recurrent (see \cite[Theorem~3.4]{Zhao98}).

We usually use an approximation, the level-increment (LI) truncation approximation, to implement Ramaswami's recursion. Ramaswami's recursion (\ref{eq:Ramaswami}) and auxiliary equations (\ref{eqn:recursion_G})--(\ref{defn:R-matrices}) have the infinite sums originated from the infinite sequences $\{\vc{A}(k)\}$ and $\{\vc{B}(k)\}$. The infinite sums are obstacles to implementing Ramaswami's recursion. Therefore, to remove the obstacles, we usually truncate the infinite sequences $\{\vc{A}(k)\}$ and $\{\vc{B}(k)\}$, which is the level-increment (LI) truncation approximation to the M/G/1-type Markov chain: For $N \in \bbN$,
\begin{align*}
\vc{P}^{(N)}
:=
\bordermatrix{
& \bbL_0
& \bbL_1
& \bbL_2
& \bbL_3
& \cdots \cr
\bbL_0
& \vc{B}^{(N)}(0)
& \vc{B}^{(N)}(1)
& \vc{B}^{(N)}(2)
& \vc{B}^{(N)}(3)
& \cdots \cr
\bbL_1
& \vc{B}^{(N)}(-1)
& \vc{A}^{(N)}(0)
& \vc{A}^{(N)}(1)
& \vc{A}^{(N)}(2)
& \cdots \cr
\bbL_2
& \vc{O}
& \vc{A}^{(N)}(-1)
& \vc{A}^{(N)}(0)
& \vc{A}^{(N)}(1)
& \cdots \cr
\bbL_3
& \vc{O}
& \vc{O}
& \vc{A}^{(N)}(-1)
& \vc{A}^{(N)}(0)
& \cdots
\cr
~\vdots
& \vdots
& \vdots
& \vdots
& \vdots
& \ddots
},
\end{align*}
where
\begin{eqnarray*}
\vc{A}^{(N)}(k)
&=&
\left\{
\begin{array}{ll}
\vc{A}(k),			& k \in \{-1,0,1,\dots,N-1\},
\\
\ol{\vc{A}}(N-1):=\dm\sum_{\ell=N}^{\infty}\vc{A}(\ell), 	& k=N,
\\
\vc{O}, 			& k \in \{N+1,N+2,\dots\},
\end{array}
\right.
\\
\vc{B}^{(N)}(k)
&=&
\left\{
\begin{array}{ll}
\vc{B}(k),			& k \in \{-1,0,1,\dots,N-1\},
\\
\ol{\vc{B}}(N-1):=\dm\sum_{\ell=N}^{\infty}\vc{B}(\ell), 	& k=N,
\\
\vc{O}, 			& k \in \{N+1,N+2,\dots\}.
\end{array}
\right.
\end{eqnarray*}
We refer to $\vc{P}^{(N)}$ as the LI truncation approximation to $\vc{P}$.

Assumption~\ref{assum1} ensures that the LI truncation approximation $\vc{P}^{(N)}$ has the unique stationary distribution vector, denoted by $\vc{\pi}^{(N)} := (\pi^{(N)}(k, i))_{(k,i)\in\bbS}$, as shown in Proposition~\ref{prop:convergence_piN_to_pi}. Based on Proposition~\ref{prop:convergence_piN_to_pi}, we refer to $\vc{\pi}^{(N)}$ as the LI truncation approximation to $\vc{\pi}$.
\begin{PROP}[{}{\cite[Proposition~3.1 and Theorem 4.3]{Ouchi-Masuyama21}}]\label{prop:convergence_piN_to_pi}
If Assumption~\ref{assum1} holds, $\vc{P}^{(N)}$ has the unique stationary distribution vector $\vc{\pi}^{(N)}$, which converges to the (original) stationary distribution $\vc{\pi}$ of $\vc{P}$ in total-variation norm, that is,
\begin{equation*}
\lim_{N \to \infty}\|\vc{\pi}^{(N)} - \vc{\pi}\| = 0,
\end{equation*}
where for any vector $\vc{x} := (x_i)_{i\in\calX}$, $\|\vc{x}\|$ denotes the total-variation norm of $\vc{x}$, that is,
\[
\|\vc{x}\| = \sum_{i \in \calX} |x_i|.
\]
\end{PROP}

The LI truncation approximation $\vc{\pi}^{(N)}$ to $\vc{\pi}$ satisfies Ramaswami's recursion for the M/G/1-type stochastic matrix $\vc{P}^{(N)}$. More specifically, the vectors $\vc{\pi}^{(N)}(k) := (\pi^{(N)}(k, i))_{(k,i)\in\bbL_k}$, $k \in \bbZ_+$ are determined by the recursion obtained by replacing $\{\vc{A}(k)\}$ and $\{\vc{B}(k)\}$ with $\{\vc{A}^{(N)}(k)\}$ and $\{\vc{B}^{(N)}(k)\}$. To save space, we omit the details (see \cite[Section 1]{Ouchi-Masuyama21-2}).

\section{Main Results}
\label{sec:subgeometric}

This section contains the main results of this paper. First, we introduce two additional assumptions and then provide the existing asymptotic formulas for the original stationary distribution vector and its LI truncation approximation. With the asymptotic formulas, we prove the main theorem of this paper.

We begin with making two additional assumptions.
\begin{ASSU}\label{assum:aperiodic}
The single communication class of $\vc{G}$ is aperiodic.
\end{ASSU}
\begin{ASSU}\label{assum3}
Let
\begin{alignat*}{2}
\ool{\vc{A}}(k)
&= \sum_{\ell = k + 1}^\infty \ol{\vc{A}}(\ell),
& \qquad
\ool{\vc{B}}(k)
&= \sum_{\ell = k + 1}^\infty \ol{\vc{B}}(\ell).
\end{alignat*}
There exists a distribution function $F$ on $\bbR_+:=[0,\infty)$ such that
\[
\lim_{N \to \infty}
{
\ool{\vc{A}}(N)\vc{e}
\over
\ol{F}(N)
} = \vc{c}_A,
\quad
\lim_{N \to \infty}
{
\ool{\vc{B}}(N)\vc{e}
\over
\ol{F}(N)
} = \vc{c}_B,
\]
where $\vc{c}_A \ge \vc{0}$ and $\vc{c}_B \ge \vc{0}$ are $M_1$- and $M_0$-dimensional finite column vectors, respectively, and either of them is a non-zero vector.
\end{ASSU}

\smallskip

Before presenting our main theorem, we provide Proposition~\ref{prop:tail_pi} below, which together with Proposition~\ref{prop:high-order_long-tailed} is key to proving the theorem. The proposition does not necessarily require Assumption~\ref{assum:aperiodic}.
\begin{PROP}[{}{\cite[Theorem 3.1]{Masu16-ANOR}}, {\cite[Theorem 5.2]{Ouchi-Masuyama21}}] \label{prop:tail_pi}
Suppose that Assumptions \ref{assum1} and \ref{assum3} hold, and let $\ol{\vc{\pi}}(k) = \sum_{\ell=k+1}^\infty \vc{\pi}(\ell)$ for $k \in \bbZ_+$. If $F$ is long-tailed (i.e., $F \in \calL$; see Definition~\ref{defn-long-tailed}), then
\begin{alignat}{2}
\lim_{N \to \infty}
{
\vc{\pi}^{(N)}(k) - \vc{\pi}(k)
\over
\ol{F}(N)
}
&=
{\vc{\pi}(0)\vc{c}_B + \ol{\vc{\pi}}(0)\vc{c}_A
\over
-\sigma
}\vc{\pi}(k)>\vc{0}, &\qquad k &\in \bbZ_+.
\label{eq:subexp-01}
\end{alignat}
Furthermore, if $F$ is subexponential (that is, $F \in \calS \subsetneq \calL$; see Definition~\ref{defn-long-tailed}), then
\begin{equation}
\lim_{N\to\infty} {\ol{\vc{\pi}}(N) \over \ol{F}(N)} = {\vc{\pi}(0)\vc{c}_B + \ol{\vc{\pi}}(0)\vc{c}_A \over -\sigma}\vc{\varpi}.
\label{eq:tail_pi}
\end{equation}

\end{PROP}

\begin{REM}
The conditions of Proposition~\ref{prop:tail_pi} do not necessarily imply that the mean of $F$ is finite. Let $H$ denote a long-tailed distribution $H$ on $\bbR_+$ such that $\ol{H}(x) := 1 - H(x) = (x+1)^{-\gamma}$ for $x \in \bbR_+$, where $1 < \gamma \le 2$. Suppose that
\begin{align}
\lim_{N \to \infty}
{
\ol{\vc{A}}(N)\vc{e}
\over
\ol{H}(N)
} = \vc{c}_A^*,
\quad
\lim_{N \to \infty}
{
\ol{\vc{B}}(N)\vc{e}
\over
\ol{H}(N)
} = \vc{c}_B^*,
\label{eqn:22_1217-01}
\end{align}
where $\vc{c}_A^* \ge \vc{0}$ and $\vc{c}_B^* \ge \vc{0}$ are $M_1$- and $M_0$-dimensional finite column vectors, respectively, and either of them is a non-zero vector. By definition, the mean of $H$ is equal to $(\gamma-1)^{-1} < \infty$ and thus the present setting is compatible with Assumption~\ref{assum1}. Furthermore, let $H_I$ denote the integrated tail distribution of $F$ on $\bbR_+$ and let $\ol{H}_I(x)$ denote
\begin{align*}
\ol{H}_I(x) = (\gamma - 1) \dm\int_{x}^{\infty} \ol{H}(t) dt
= (x+1)^{-\gamma+1},\qquad x \in \bbR_+.
\end{align*}
Thus, the mean of $H_I$ is infinite. Furthermore, (\ref{eqn:22_1217-01}) yields
\begin{align*}
\lim_{N \to \infty}
{
\ool{\vc{A}}(N)\vc{e}
\over
\ol{H}_I(N)
} = {\vc{c}_A^* \over \gamma - 1},
\quad
\lim_{N \to \infty}
{
\ool{\vc{B}}(N)\vc{e}
\over
\ol{H}_I(N)
} = {\vc{c}_B^* \over \gamma - 1}.
\end{align*}
As a result, Assumption~\ref{assum3} holds with
\begin{align*}
F = H_I, \qquad
\vc{c}_A = {\vc{c}_A^* \over \gamma - 1},
\qquad
\vc{c}_B = {\vc{c}_B^* \over \gamma - 1}.
\end{align*}
\end{REM}

\smallskip

The following is the main theorem of this paper.
\begin{THM}\label{th:subgeo_tvd}
Suppose that Assumptions~\ref{assum1}, \ref{assum:aperiodic}, and \ref{assum3} hold. Suppose that $F$ is $p$th-order long-tailed for some $p > 1$ (that is, $F \in \calL^p$; see Definition~\ref{defn-high-order_long-tailed}). We then have
\begin{equation}
\lim_{N\to\infty} {\|\vc{\pi}^{(N)}-\vc{\pi}\| \over \ol{F}(N)}
=
{\vc{\pi}(0)\vc{c}_B + \ol{\vc{\pi}}(0)\vc{c}_A \over -\sigma} > 0.
\label{eq:subgeo_tvd}
\end{equation}
Furthermore, if $F$ is subexponential (that is, $F \in \calS$),
\begin{equation}
\lim_{N\to\infty} {\|\vc{\pi}^{(N)}-\vc{\pi}\| \over \ol{\vc{\pi}}(N)\vc{e}} = 1.
\label{eq:subgeo_tvd-2}
\end{equation}
\end{THM}
\begin{proof}
We provide the proof of (\ref{eq:subgeo_tvd}) in Appendix~\ref{sec:proof_of_theorem}. From \eqref{eq:tail_pi} and \eqref{eq:subgeo_tvd}, we obtain
\begin{align*}
\lim_{N \to \infty}
\frac{\|\vc{\pi}^{(N)} - \vc{\pi}\|}{\ol{\vc{\pi}}(N)\vc{e}}
&= \lim_{N \to \infty}
\frac{\|\vc{\pi}^{(N)} - \vc{\pi}\|}{\ol{F}(N)} \cdot
\frac{\ol{F}(N)}{\ol{\vc{\pi}}(N)\vc{e}}
\\
&= \frac{\vc{\pi}(0)\vc{c}_B + \ol{\vc{\pi}}(0)\vc{c}_A}{-\sigma}
\cdot
\frac{-\sigma}{\vc{\pi}(0)\vc{c}_B + \ol{\vc{\pi}}(0)\vc{c}_A}= 1,
\end{align*}
which yields \eqref{eq:subgeo_tvd-2}.
\end{proof}

\begin{REM}
Equation \eqref{eq:subgeo_tvd-2} requires \eqref{eq:tail_pi}  and thus $F \in \calS$.
\end{REM}

We comment on the total-variation convergence formula (\ref{eq:subgeo_tvd}) in Theorem~\ref{th:subgeo_tvd}, compared with the level-wise convergence formula (\ref{eq:subexp-01}). Equation (\ref{eq:subgeo_tvd}) shows that the convergence of (\ref{eq:subexp-01}) is uniform over $k \in \bbZ_+$. However, this uniform convergence is not obvious from the level-wise convergence formula (\ref{eq:subexp-01}). With the definition of the total variation norm, we can rewrite the left-hand side of (\ref{eq:subgeo_tvd}) as
\begin{align}
\lim_{N\to\infty} {\|\vc{\pi}^{(N)}-\vc{\pi}\| \over \ol{F}(N)}
&= \lim_{N\to\infty} \sum_{(k,i) \in \bbS}
{|\pi^{(N)}(k,i) - \pi(k,i)| \over \ol{F}(N)}.
\label{eqn:220813-01}
\end{align}
Therefore, if we are allowed to change the order between the limit and infinite sum in (\ref{eqn:220813-01}), then we can obtain the following result by substituting (\ref{eq:subexp-01}) into (\ref{eqn:220813-01}):
\begin{align*}
\lim_{N\to\infty} {\|\vc{\pi}^{(N)}-\vc{\pi}\| \over \ol{F}(N)}
&=  \sum_{(k,i) \in \bbS}
\lim_{N\to\infty}
{|\pi^{(N)}(k,i) - \pi(k,i)| \over \ol{F}(N)}
\nonumber
\\
&= {\vc{\pi}(0)\vc{c}_B + \ol{\vc{\pi}}(0)\vc{c}_A
\over
-\sigma
} \sum_{(k,i) \in \bbS} \pi(k,i)
\nonumber
\\
&= {\vc{\pi}(0)\vc{c}_B + \ol{\vc{\pi}}(0)\vc{c}_A
\over
-\sigma
},
\end{align*}
which leads to (\ref{eq:subgeo_tvd}). In this way, deriving the total-variation convergence formula (\ref{eq:subgeo_tvd}) from the level-wise convergence formula (\ref{eq:subexp-01}) demands changing the order between the limit and infinite sum. Such an operation is not allowed in general.

Finally, we mention the decay speed of the total variation error $\|\vc{\pi}^{(N)} - \vc{\pi}\|$. The total-variation convergence formula (\ref{eq:subgeo_tvd}), together Assumption~\ref{assum3}, implies that $\|\vc{\pi}^{(N)} - \vc{\pi}\|$ decays at same rate as $\ool{\vc{A}}(N)\vc{e}$ and/or $\ool{\vc{B}}(N)\vc{e}$, that is, the integrated tail distribution of level increments in steady state (for detail, see \cite[Section~5]{Ouchi-Masuyama21}). In addition, the second formula (\ref{eq:subgeo_tvd-2}) shows that $\|\vc{\pi}^{(N)} - \vc{\pi}\|$ also decays at same rate as the tail probability $\ol{\vc{\pi}}(N)\vc{e}$ of the original stationary distribution. Although we could infer from the level-wise convergence formula (\ref{eq:subexp-01}) these arguments on the total variation error, they are now theoretically guaranteed by the presentation of the total-variation convergence formula.

%
\section{Concluding Remarks}
\label{sec:concluding}

The main contribution of this paper is to present the subgeometric total-variation convergence formula for the stationary distribution in M/G/1-type Markov chains. The total-variation convergence formula is a stronger result than the level-wise convergence formula in terms of convergence. However, the former requires the additional condition, the aperiodicity of the $G$-matrix. It is a future task to remove this additional condition. As mentioned in \cite[Section 5]{Ouchi-Masuyama21-2}, another future task is to derive the geometric version of Theorem \ref{th:subgeo_tvd}; that is, the geometric total-variation convergence formula for the stationary distribution in the M/G/1-type Markov chains. Successful completion of this task would bring closure to the research on the convergence speed of the LI truncation approximation for M/G/1-type Markov chains.

\appendix
\section{Long-tailed and Subexponential Distributions}
\label{sec-subexponential_dist}

This section is concerned with long-tailed and subexponential distributions. We first define the classes of long-tailed and subexponential distributions. We then define the class of high-order long-tailed distributions, which is a subclass of the long-tailed one. Finally, we provide a proposition on the high-order long-tailed distributions. The proposition applies to the proof of Theorem~\ref{th:subgeo_tvd} in Section~\ref{sec:subgeometric}.

The following is the definition of long-tailed and subexponential distributions.%
\begin{DEF}[{}{\cite[Definitions 2.21 and 3.1]{Foss13}}]\label{defn-long-tailed}
\hfill\begin{enumerate}
\item A nonnegative random variable $Y$ and its distribution function $F$ are said to be {\it long-tailed} if and only if
\begin{align*}
\lim_{x \to \infty} {\PP(Y > x + y) \over \PP(Y > x)} = 1, \quad \forall y > 0.
\end{align*}
The class of long-tailed distributions is denoted by $\calL$.
\item A nonnegative random variable $Y$ and its distribution function $F$ are said to be {\it subexponential} if and only if
\begin{align*}
\lim_{x \to \infty} {\PP(Y_1 + Y_2 > x) \over \PP(Y > x)} = 2,
\end{align*}
where $Y_1$ and $Y_2$ are independent copies of $Y$. The class of subexponential distributions is denoted by $\calS$.
\end{enumerate}
\end{DEF}

\medskip

\begin{REM}
The class of long-tailed distribution includes that of subexponential distribution; that is, $\calS \subsetneq \calL$ (see \cite[Lemma 3.2]{Foss13}).
\label{rem1}
\end{REM}

\medskip

The following is the definition of high-order long-tailed distributions.
\begin{DEF}[{}{\cite[Definition 1.1]{Masu13}}]\label{defn-high-order_long-tailed}
A nonnegative random variable $Y$ and its distribution function $F$ are said to be {\it $p$th-order long-tailed} ($p \ge 1$) if and only if $Y^{1/p}$ is long-tailed, that is,
\begin{align*}
\lim_{x \to \infty} {\PP(Y^{1/p} > x + y) \over \PP(Y^{1/p} > x)} = 1, \quad \forall y > 0.
\end{align*}
The class of $p$th-order long-tailed distributions is denoted by $\calL^p$. Clearly, $\calL^1 = \calL$.
\end{DEF}

\medskip

\begin{REM}\label{rem:high-order_long-tailed}
The inclusion relation $\calL^{p_2} \subset \calL^{p_1} \subset \calL$ holds for $1 < p_1 < p_2$ (see \cite[Lemma~A.1]{Masu13}).
\end{REM}

\medskip

The following proposition contributes to the proof of Theorem~\ref{th:subgeo_tvd}.
\begin{PROP}[{}{\cite[Lemma~A.2]{Masu13}}]\label{prop:high-order_long-tailed}
For any $p \ge 1$, a nonnegative random variable $Y$ is $p$th-order long-tailed if and only if
\begin{equation}
\lim_{x\to\infty} {\PP(Y > x - \xi x^{1 - 1/p}) \over \PP(Y > x)} = 1
\quad \mbox{for some (and thus all) $\xi \in (-\infty,\infty)$}.
\label{eq:high-order_long-tailed}
\end{equation}
\end{PROP}

\begin{REM}
Let $F_I$ denote the integrated tail distribution of $F$ on $\bbR_+$, and let $\ol{F}_I(x) = 1 - F_I(x)$ for $x \in \bbR_+$, that is,
\begin{align*}
\ol{F}_I(x) = { \dm\int_{x}^{\infty} \ol{F}(t) dt \over \dm\int_{0}^{\infty} \ol{F}(t) dt}.
\end{align*}
Furthermore, suppose that $F \in \calL^{p}$ for some $p \ge 1$. Since $F \in \calL$ (see Remark~\ref{rem:high-order_long-tailed}), we have (see, e.g., \cite[Lemma~2.25]{Foss13})
\begin{align*}
\lim_{x\to\infty} {\ol{F}(x) \over \ol{F}_I(x)} = 0.
\end{align*}
However, $F_I$ also belongs to class $\calL^p$. Indeed, using l'H\^{o}pital's rule, we obtain
\begin{align*}
\lim_{x\to\infty} {1 - F_I(x - x^{1 - 1/p}) \over 1 - F_I(x)}
&= \lim_{x\to\infty}
{
\dm\int_{x - x^{1 - 1/p}}^{\infty} \ol{F}(t) dt
\over
\dm\int_{x}^{\infty} \ol{F}(t) dt
}
=\lim_{x\to\infty}
{\ol{F}(x - x^{1 - 1/p}) \over \ol{F}(x)}
=1,
\end{align*}
and thus $F_I \in \calL^p$.
\end{REM}
\section{Proof of (\ref{eq:subgeo_tvd}) in Theorem~\ref{th:subgeo_tvd}}
\label{sec:proof_of_theorem}

We introduce the notation needed in this section. For any matrix $\vc{V}:=(V_{i,j})$, let $|\vc{V}|$ denote the matrix whose $(i,j)$th element is equal to $|V_{i,j}|$. Similarly, for any (row or column) vector $\vc{v}:=(v_i)$, let $|\vc{v}|$ denote the vector whose $i$th element is equal to $|v_i|$. Using this notation, we have
\begin{align*}
|\vc{\pi}^{(N)} - \vc{\pi}| \vc{e}
= \sum_{(k,i) \in \bbS} |\pi^{(N)}(k,i) - \pi(k,i)|
= \|\vc{\pi}^{(N)}-\vc{\pi}\|.
\end{align*}

To prove \eqref{eq:subgeo_tvd}, we begin with showing that it holds if
\begin{equation}
\sup_{N \in \bbN} \sum_{k=0}^{N - \lfloor N^{1- 1/p} \rfloor}
{|\vc{\pi}^{(N)}(k) - \vc{\pi}(k)|\vc{e} \over  \ol{F}(N) } < \infty,
\label{eq:sup_piN-pi_head}
\end{equation}
where $\lfloor\, \cdot \,\rfloor$ denotes the floor function. For simplicity, we use $\theta := 1 - 1/p$ throughout this section. Proposition~\ref{prop:convergence_piN_to_pi} implies that there exists a constant $\delta > 0$ such that
\[
\sum_{k=m+1}\vc{\pi}^{(N)}(k) \le (1 + \delta)\ol{\vc{\pi}}(m), \qquad m \in \bbZ_+, \quad N \in \bbN,
\]
and thus
\begin{alignat}{2}
\sum_{k = N - \lfloor N^\theta \rfloor + 1}^\infty {|\vc{\pi}^{(N)}(k) - \vc{\pi}(k)|\vc{e} \over \ol{F}(N)}
&\le \sum_{k = N - \lfloor N^\theta \rfloor + 1}^\infty {(\vc{\pi}^{(N)}(k) + \vc{\pi}(k))\vc{e} \over \ol{F}(N)}
\nonumber
\\
&\le (2 + \delta) {\ol{\vc{\pi}}(N - \lfloor N^\theta \rfloor)\vc{e} \over \ol{F}(N)},& &\qquad N \in \bbN.
\label{eq:tail_piN}
\end{alignat}
It follows from \eqref{eq:tail_pi}, \eqref{eq:high-order_long-tailed}, and $F \in \calS \cap \calL^p$ that
\begin{align}
\lim_{N \to \infty} {\ol{\vc{\pi}}(N- \lfloor N^\theta \rfloor)\vc{e} \over \ol{F}(N)}
&= \lim_{N \to \infty} {\ol{\vc{\pi}}(N- \lfloor N^\theta \rfloor)\vc{e} \over \ol{F}(N- \lfloor N^\theta \rfloor)} {\ol{F}(N- \lfloor N^\theta \rfloor) \over \ol{F}(N)}
\nonumber
\\
&= {\vc{\pi}(0)\vc{c}_B + \ol{\vc{\pi}}(0)\vc{c}_A \over -\sigma} < \infty.
\label{eq:tail_pi-2}
\end{align}
Combining \eqref{eq:tail_piN} and \eqref{eq:tail_pi-2} yields
\begin{equation}
\sup_{N \in \bbN} \sum_{k = N - \lfloor N^\theta \rfloor + 1}^\infty {|\vc{\pi}^{(N)}(k) - \vc{\pi}(k)|\vc{e} \over \ol{F}(N)} < \infty.
\label{eq:tail_piN-pi}
\end{equation}
From (\ref{eq:sup_piN-pi_head}) and (\ref{eq:tail_piN-pi}), we have
\begin{align*}
\sup_{N \in \bbN} \sum_{k =0}^\infty {|\vc{\pi}^{(N)}(k) - \vc{\pi}(k)|\vc{e} \over \ol{F}(N)} < \infty.
\end{align*}
Therefore, using \eqref{eq:subexp-01} and the dominated convergence theorem, we obtain
\begin{align*}
\lim_{N\to\infty} {\|\vc{\pi}^{(N)}-\vc{\pi}\| \over \ol{F}(N)} &= \sum_{k=0}^\infty \lim_{N\to\infty} {|\vc{\pi}^{(N)}(k)-\vc{\pi}(k)|\vc{e} \over \ol{F}(N)}
\\
&= \sum_{k=0}^\infty {\vc{\pi}(0)\vc{c}_B + \ol{\vc{\pi}}(0)\vc{c}_A \over -\sigma}\vc{\pi}(k)\vc{e}
\\
&= {\vc{\pi}(0)\vc{c}_B + \ol{\vc{\pi}}(0)\vc{c}_A \over -\sigma},
\end{align*}
and thus \eqref{eq:subgeo_tvd} holds.

Next, we estimate the left-hand side of \eqref{eq:sup_piN-pi_head} by using the difference formula for $\vc{\pi}^{(N)}(k) - \vc{\pi}(k)$ \cite[Lemma~4.2]{Ouchi-Masuyama21}. To this end, we introduce some matrices. Let $\vc{F}_+(k; \ell) := (F_+(k,i; \ell,j))$ and $\vc{H}_{\vc{\alpha}}(k;\ell) := (H_{\vc{\alpha}}(k,i; \ell,j))$ for $k$, $\ell \in \bbZ_+$ denote $M_{k\vmin 1} \times M_{\ell\vmin 1}$ matrices such that
\begin{align}
F_+(k, i; \ell, j) &= \EE_{(k,i)} \left[ \sum_{n=0}^{T_0 - 1} \one((X_n, J_n) = (\ell, j))\right],
\label{eq:def-F_+}
\\
H_{\vc{\alpha}}(k,i;\ell,j) &= \EE_{(k,i)}\!
\left[
\sum_{n=0}^{T_{\{\vc{\alpha}\}}-1} \one((X_{n},J_{n})=(\ell,j))
\right]
- \pi(\ell,j)\EE_{(k,i)}
\left[ T_{\{\vc{\alpha}\}} \right],
\label{eq:def-H}
\end{align}
where $\one(\cdot)$ denotes the indicator function and $\vc{\alpha} \in \bbS$ is a fixed state. Let
\begin{alignat}{2}
\vc{S}(k) &= (\vc{I}-\vc{\Phi}(0))^{-1}\vc{B}(-1)\vc{H}_{\vc{\alpha}}(0; k)
+ \vc{G}(\vc{I} -\vc{A} -\ol{\vc{m}}_A\vc{g})^{-1}\vc{e}\vc{\pi}(k),
&\quad
k &\in \mathbb{Z}_+.
\label{eq:defS}
\end{alignat}
Using these matrices, we express $\vc{\pi}^{(N)}(k) - \vc{\pi}(k)$, $k \in \{0,1,\dots,N\}$ and $N \in \bbN$ as follows:
\begin{align}
&\vc{\pi}^{(N)}(k) - \vc{\pi}(k)
\nonumber
\\
&~~= \vc{\pi}^{(N)}(0)
\left[
{1 \over -\sigma}
\ool{\vc{B}}(N - 1)\vc{e}
\vc{\pi}(k)
+
\sum_{n=N+1}^{\infty}
\vc{B}(n) ( \vc{G}^{N-k} - \vc{G}^{n-k} )
\vc{F}_+(k;k)
\right.
\nonumber
\\
&\qquad\qquad\qquad\qquad\quad
\Biggl.
+
\sum_{n=N+1}^{\infty}
\vc{B}(n)( \vc{G}^{N-1} - \vc{G}^{n-1} )
\vc{S}(k)
\Biggr]
\nonumber
\\
&~~~~
+ \sum_{\ell=1}^{\infty}
\vc{\pi}^{(N)}(\ell)
\left[
{1 \over -\sigma}
\ool{\vc{A}}(N - 1)\vc{e}
\vc{\pi}(k)
+ \sum_{n=N+1}^{\infty}
\vc{A}(n) (\vc{G}^{N + \ell - k} - \vc{G}^{n + \ell - k})
\vc{F}_+(k;k)
\right.
\qquad
\nonumber
\\
&\qquad\qquad\qquad\qquad\quad
\Biggl.
+ \sum_{n=N+1}^{\infty}
\vc{A}(n) (\vc{G}^{N + \ell - 1} - \vc{G}^{n + \ell - 1})
\vc{S}(k)
\Biggr].
\label{eq:piN(k)-pi(k)}
\end{align}
It follows from \eqref{eq:piN(k)-pi(k)} that
\begin{align*}
&\sum_{k=0}^{N- \lfloor N^\theta \rfloor}{|\vc{\pi}^{(N)}(k) - \vc{\pi}(k)|\vc{e} \over \ol{F}(N)}
\\
&~~\le \sum_{k=0}^{N- \lfloor N^\theta \rfloor}{\vc{\pi}^{(N)}(0) \over \ol{F}(N)}
\left[
{1 \over -\sigma}
\ool{\vc{B}}(N - 1)\vc{e}\vc{\pi}(k)\vc{e}
+
\sum_{n=N+1}^{\infty}
\vc{B}(n) \left| \vc{G}^{N-k} - \vc{G}^{n-k} \right|
\vc{F}_+(k;k)\vc{e}
\right.
\\
&\qquad\qquad\qquad\qquad\quad
\Biggl.
+
\sum_{n=N+1}^{\infty}
\vc{B}(n)\left| \vc{G}^{N-1} - \vc{G}^{n-1} \right|
\left| \vc{S}(k) \right|\vc{e}
\Biggr]
\\
&~~~~
+ \sum_{k=0}^{N- \lfloor N^\theta \rfloor}\sum_{\ell=1}^{\infty}
{\vc{\pi}^{(N)}(\ell) \over \ol{F}(N)}
\left[
{1 \over -\sigma}
\ool{\vc{A}}(N - 1)\vc{e}
\vc{\pi}(k)\vc{e}
+ \sum_{n=N+1}^{\infty}
\vc{A}(n) \left| \vc{G}^{N + \ell - k} - \vc{G}^{n + \ell - k} \right|
\vc{F}_+(k;k)\vc{e}
\right.
\\
&\qquad\qquad\qquad\qquad\quad
\Biggl.
+ \sum_{n=N+1}^{\infty}
\vc{A}(n) \left| \vc{G}^{N + \ell - 1} - \vc{G}^{n + \ell - 1} \right|
\left| \vc{S}(k) \right|\vc{e}
\Biggr], \qquad N \in \bbN.
\end{align*}
Therefore, we obtain an upper bound for the left-hand side of \eqref{eq:sup_piN-pi_head}:
\begin{equation}
\sum_{k=0}^{N- \lfloor N^\theta \rfloor}{|\vc{\pi}^{(N)}(k) - \vc{\pi}(k)|\vc{e} \over \ol{F}(N)}
\le C_1(N) + C_2(N) + C_3(N),
\label{eq:head_piN-pi_2}
\end{equation}
where
\begin{align}
C_1(N) &= \sum_{k=0}^{N- \lfloor N^\theta \rfloor} {1 \over -\sigma\ol{F}(N)}
\left[
\vc{\pi}^{(N)}(0)\ool{\vc{B}}(N-1)\vc{e} + \sum_{\ell=1}^{\infty}\vc{\pi}^{(N)}(\ell)\ool{\vc{A}}(N-1)\vc{e}\right]\vc{\pi}(k)\vc{e},
\label{defn:C_1(N)}
\\
C_2(N) &= \sum_{k=0}^{N- \lfloor N^\theta \rfloor} {1 \over \ol{F}(N)} \Biggl[
\vc{\pi}^{(N)}(0)\sum_{n=N+1}^\infty \vc{B}(n) \left|\vc{G}^{N-k} - \vc{G}^{n-k}\right|\Biggr.
\nonumber
\\
&\Biggl.\qquad\qquad\qquad+ \sum_{\ell=1}^\infty \vc{\pi}^{(N)}(\ell) \sum_{n=N+1}^\infty \vc{A}(n) \left|\vc{G}^{N-k} - \vc{G}^{n-k}\right| \vc{G}^{\ell}
\Biggr]\vc{F}_+(k;k)\vc{e},
\label{defn:C_2(N)}
\\
C_3(N) &= \sum_{k=0}^{N - \lfloor N^\theta \rfloor} {1 \over \ol{F}(N)}\Biggl[
\vc{\pi}^{(N)}(0)\sum_{n=N+1}^\infty \vc{B}(n) \left|\vc{G}^{N-1} - \vc{G}^{n-1}\right| \Biggr.
\nonumber
\\
&\Biggl.\qquad\qquad\qquad+ \sum_{\ell=1}^\infty \vc{\pi}^{(N)}(\ell) \sum_{n=N+1}^\infty \vc{A}(n) \left|\vc{G}^{N-1} - \vc{G}^{n-1}\right| \vc{G}^{\ell}
\Biggr]|\vc{S}(k)|\vc{e}.
\label{defn:C_3(N)}
\end{align}

The upper bound (\ref{eq:head_piN-pi_2}) enables us to reduce the proof of (\ref{eq:sup_piN-pi_head}) to those of the following equations:
\begin{align}
\lim_{N\to\infty} C_1(N)
&= {1 \over -\sigma}
\left[
\vc{\pi}(0) \vc{c}_B
+ \ol{\vc{\pi}}(0) \vc{c}_A
\right],
\label{eq:C_bounded}
\\
\lim_{N\to\infty} C_2(N) &= 0,
\label{eq:D_bounded}
\\
\lim_{N\to\infty} C_3(N) &= 0.
\label{eq:E_bounded}
\end{align}
In what follows, we show that these equations hold to achieve our goal, to prove (\ref{eq:sup_piN-pi_head}) and thus to prove Theorem~\ref{th:subgeo_tvd}.

First, we prove (\ref{eq:C_bounded}). Note that
\begin{align*}
\vc{\pi}^{(N)}(0)\vc{e} &\le 1, \quad
\sum_{\ell=1}^{\infty}\vc{\pi}^{(N)}(\ell)\vc{e} \le 1, \quad
\sum_{k=0}^{N- \lfloor N^\theta \rfloor}\vc{\pi}(k)\vc{e} \le 1 \quad
\mbox{for all $N \in \bbN$}.
\end{align*}
In addition, Assumption \ref{assum3} yields
\begin{subequations}\label{eqn:220830-01}
\begin{align}
\lim_{N\to\infty} {\ool{\vc{A}}(N-1)\vc{e} \over \ol{F}(N)}
&= \lim_{N\to\infty} {\ool{\vc{A}}(N-1)\vc{e} \over \ol{F}(N-1)} {\ol{F}(N) \over \ol{F}(N-1)} = \vc{c}_A,
\\
\lim_{N\to\infty} {\ool{\vc{B}}(N-1)\vc{e} \over \ol{F}(N)}
&= \lim_{N\to\infty} {\ool{\vc{B}}(N-1)\vc{e} \over \ol{F}(N-1)} {\ol{F}(N) \over \ol{F}(N-1)} = \vc{c}_B.
\end{align}
\end{subequations}
Applying the dominated convergence theorem, (\ref{eqn:220830-01}), and Proposition~\ref{prop:convergence_piN_to_pi} to (\ref{defn:C_1(N)}), we obtain
\begin{align*}
\lim_{N\to\infty} C_1(N)
&= {1 \over -\sigma}
\sum_{k=0}^{\infty}
\left[
\lim_{N\to\infty}
\vc{\pi}^{(N)}(0) { \ool{\vc{B}}(N-1)\vc{e} \over \ol{F}(N) }
+ \sum_{\ell=1}^{\infty}
\lim_{N\to\infty} \vc{\pi}^{(N)}(\ell){ \ool{\vc{A}}(N-1)\vc{e} \over \ol{F}(N)}
\right]
\vc{\pi}(k)\vc{e}
\nonumber
\\
&= {1 \over -\sigma}
\left[
\vc{\pi}(0) \vc{c}_B
+ \sum_{\ell=1}^{\infty}
\vc{\pi}(\ell) \vc{c}_A
\right]
\sum_{k=0}^{\infty}\vc{\pi}(k)\vc{e}
\nonumber
\\
&= {1 \over -\sigma}
\left[
\vc{\pi}(0) \vc{c}_B
+ \ol{\vc{\pi}}(0) \vc{c}_A
\right],
\end{align*}
which shows that (\ref{eq:C_bounded}) holds.

Next, we prove (\ref{eq:D_bounded}). It follows from \eqref{eq:def-F_+} that
\begin{equation*}
F_+(k, i; \ell, j) \le \EE_{(k,i)}[T_0], \qquad (k, i), (\ell, j) \in \bbS.
\end{equation*}
This inequality and \cite[Lemma~3.5]{Masu21-M/G/1-Subexp} imply that there exists some $\psi_0 > 0$ such that
\begin{equation*}
\vc{F}_+(k; k) \vc{e} \le \psi_0 N \vc{e} \quad
\mbox{for all $N \in \bbN$ and $k \in \{0,1,\dots,N\}$.}
\end{equation*}
Applying this to (\ref{defn:C_2(N)}) and using $\vc{G}\vc{e}=\vc{e}$, we obtain
\begin{align}
C_2(N) &\le \psi_0 N
\sum_{k=0}^{N- \lfloor N^\theta \rfloor} {1 \over \ol{F}(N)} \Biggl[
\vc{\pi}^{(N)}(0)\sum_{n=N+1}^\infty \vc{B}(n) \left|\vc{G}^{N-k} - \vc{G}^{n-k}\right|\vc{e}
\Biggr.
\nonumber
\\
&\Biggl.\qquad\qquad\qquad{}
+ \sum_{\ell=1}^\infty \vc{\pi}^{(N)}(\ell) \sum_{n=N+1}^\infty \vc{A}(n) \left|\vc{G}^{N-k} - \vc{G}^{n-k}\right|\vc{e}
\Biggr].
\label{eqn:220830-02}
\end{align}
Furthermore, Assumption~\ref{assum:aperiodic} implies (see, e.g., \cite[Theorem~8.5.1]{Horn13}) that there exists some $\varepsilon > 0$ such that, for all $m \in \bbZ_+$,
\begin{equation*}
\|\vc{G}^m - \vc{e}\vc{g} \| \le C_G (1+\varepsilon)^{-m}
\quad \mbox{for some $C_G > 0$ and $\varepsilon > 0$.}
\end{equation*}
Therefore, there exists some $\psi_1 > 0$ such that
\begin{align}
\left|\vc{G}^{N-k} - \vc{G}^{n-k}\right| \vc{e}
\le \psi_1 (1+\varepsilon)^{-N^{\theta}} \vc{e}, \quad N\in\bbN, ~n \in \bbZ_{\geqslant N+1},~0 \le k \le N- \lfloor N^\theta \rfloor.
\label{eq:GN-Gn}
\end{align}
Evaluating the right-hand side of (\ref{eqn:220830-02}) by (\ref{eq:GN-Gn}), we have
\begin{align}
C_2(N)
&\le \psi_0 \psi_1 N(1+\varepsilon)^{-N^{\theta}}
\sum_{k=0}^{N- \lfloor N^\theta \rfloor} {1 \over \ol{F}(N)}
\left[
\vc{\pi}^{(N)}(0)\sum_{n=N+1}^\infty \vc{B}(n)\vc{e}
+ \sum_{\ell=1}^\infty \vc{\pi}^{(N)}(\ell) \sum_{n=N+1}^\infty \vc{A}(n) \vc{e}
\right]
\nonumber
\\
&=\psi_0 \psi_1 N(N-\lfloor N^\theta \rfloor + 1)(1+\varepsilon)^{-N^{\theta}}
\nonumber
\\
&\qquad\qquad\times\left[
\vc{\pi}^{(N)}(0) {\ol{\vc{B}}(N)\vc{e} \over \ol{F}(N)}
+ \sum_{\ell=1}^\infty \vc{\pi}^{(N)}(\ell){\ol{\vc{A}}(N)\vc{e} \over \ol{F}(N)}
\right].
\label{eqn:220830-03}
\end{align}
Assumption~\ref{assum3} ensures that
\begin{subequations}\label{eq:ol_AB_F}
\begin{align}
\lim_{N\to\infty}{\ol{\vc{A}}(N)\vc{e} \over \ol{F}(N)} = \lim_{N\to\infty}{\ool{\vc{A}}(N-1)\vc{e} - \ool{\vc{A}}(N)\vc{e} \over \ol{F}(N)} = \vc{0},
\label{eq:ol_A/F}
\\
\lim_{N\to\infty}{\ol{\vc{B}}(N)\vc{e} \over \ol{F}(N)} = \lim_{N\to\infty}{\ool{\vc{B}}(N-1)\vc{e} - \ool{\vc{B}}(N)\vc{e} \over \ol{F}(N)} = \vc{0}.
\label{eq:ol_B/F}
\end{align}
\end{subequations}
Combining these and (\ref{eqn:220830-03}) leads to
\begin{equation*}
\lim_{N\to\infty} C_2(N) = 0,
\end{equation*}
which shows that (\ref{eq:D_bounded}) holds.

Finally, we prove (\ref{eq:E_bounded}). It follows from \eqref{eq:defS} that, for $N \in \bbN$,
\begin{align}
\sum_{k=0}^{N - \lfloor N^\theta \rfloor} |\vc{S}(k)|\vc{e} &\le \sum_{k=0}^\infty |\vc{S}(k)|\vc{e}
\nonumber
\\
&\le (\vc{I}-\vc{\Phi}(0))^{-1}\vc{B}(-1)\sum_{k=0}^\infty |\vc{H}_{\vc{\alpha}}(0; k)|\vc{e}
+ \vc{G}\left|(\vc{I} -\vc{A} -\ol{\vc{m}}_A\vc{g})^{-1}\right|\vc{e}.
\label{eq:sum_Sk}
\end{align}
It also follows from \eqref{eq:def-H} that
\begin{align}
\sum_{(k, j)\in\bbS} |H_{\vc{\alpha}}(0,i; k, j)|
&\le \sum_{(k, j)\in\bbS} \EE_{(0,i)}\left[
\sum_{n=0}^{T_{\{\vc{\alpha}\}}-1} \one((X_n, J_n) = (k, j))\right]
+ \sum_{(k, j)\in\bbS} \pi(k, j)\EE_{(0,i)}[T_{\{\vc{\alpha}\}}]
\nonumber
\\
&= 2\EE_{(0,i)}[T_{\{\vc{\alpha}\}}] < \infty, \qquad i\in\bbM_0.
\label{eq:sum_H}
\end{align}
The combination of \eqref{eq:sum_Sk} and \eqref{eq:sum_H} implies that there exists a constant $S>0$ such that
\begin{equation}
\sum_{k=0}^{N - \lfloor N^\theta \rfloor} |\vc{S}(k)|\vc{e}
\le S\vc{e} \quad \mbox{for all $N \in \bbN$.}
\label{eqn:220830-05}
\end{equation}
Thus, substituting \eqref{eq:GN-Gn}, \eqref{eqn:220830-05}, and $\vc{G}\vc{e} = \vc{e}$ into \eqref{defn:C_3(N)} yields
\begin{align}
C_3(N) &\le S {1 \over \ol{F}(N)}\Biggl[
\vc{\pi}^{(N)}(0)\sum_{n=N+1}^\infty \vc{B}(n) \left|\vc{G}^{N-1} - \vc{G}^{n-1}\right| \vc{e}
\Biggr.
\nonumber
\\
&\Biggl.\qquad\qquad\qquad+ \sum_{\ell=1}^\infty \vc{\pi}^{(N)}(\ell) \sum_{n=N+1}^\infty \vc{A}(n) \left|\vc{G}^{N-1} - \vc{G}^{n-1}\right|\vc{e}  \Biggr]
\nonumber
\\
&\le S\psi_1 (1+\varepsilon)^{-N^{\theta}}
\left[
\vc{\pi}^{(N)}(0) {\ol{\vc{B}}(N)\vc{e} \over \ol{F}(N)}
+ \sum_{\ell=1}^\infty \vc{\pi}^{(N)}(\ell) \sum_{n=N+1}^\infty {\ol{\vc{A}}(N)\vc{e} \over \ol{F}(N)} \right].
\label{eq:ol_AB_S}
\end{align}
Applying \eqref{eq:ol_AB_F} to (\ref{eq:ol_AB_S}) results in
\begin{equation*}
\lim_{N\to\infty} C_3(N) = 0,
\end{equation*}
which shows that (\ref{eq:E_bounded}) holds. The proof is completed.
\section*{Acknowledgments}
The research of the second author was supported in part by JSPS KAKENHI Grant Number JP21K11770.



\end{document}